\documentclass[11pt]{amsart}
\usepackage[margin=1.3in]{geometry}
\usepackage[utf8]{inputenc}
\usepackage{amsmath}
\usepackage[pagebackref=true]{hyperref}
\usepackage{amstext}
\usepackage{amscd}
\usepackage{amsmath}
\usepackage{amsthm}
\usepackage{latexsym}
\usepackage{url}
\usepackage{amssymb}
\usepackage[all]{xy}
\usepackage{mathrsfs} 
\usepackage{color}
\usepackage{tikz-cd}
\usetikzlibrary{matrix,arrows,decorations.pathmorphing}
\usepackage{indentfirst}
\usepackage{MnSymbol}
\usepackage{combelow}
\usepackage{enumitem}

\newtheorem{theorem}{Theorem}[section]
\newtheorem{thmx}{Theorem}
 
\newtheorem{lemma}[theorem]{Lemma}

\newtheorem*{theorem*}{Theorem}

\theoremstyle{definition}
\newtheorem{definition}[theorem]{Definition}
\newtheorem{theorem-definition}[theorem]{Theorem-Definition}

\theoremstyle{remark}

\newtheorem{example}[theorem]{Example}

\makeatletter
\def\@endtheorem{\endtrivlist}
\makeatother

\numberwithin{equation}{section}

\begin{document}
\title[Vanishing on toric surfaces]{Vanishing on toric surfaces}
\author[Yuan Wang]{Yuan Wang}
\author[Fei Xie]{Fei Xie}
\address{Department of Mathematics, University of Utah, 155 South 1400 East, Salt Lake City, UT 84112-0090, USA}
\email{ywang@math.utah.edu}
\address{Department of Mathematics, University of California at Los Angeles, Los Angeles, CA 90095-1555, USA}
\email{feixie@math.ucla.edu}
\thanks{The first author was supported in part by the NSF research grant DMS-\#1300750 and the Simons Foundation Award \#256202. The second author was supported in part by the NSF research grant DMS-\#1160206.}
\begin{abstract}
We prove the Kawamata-Viehweg vanishing and another Kodaira-type vanishing for projective toric surfaces over arbitrary fields.
\end{abstract}
\maketitle
\section{Introduction}
It is well known that Kodaira vanishing does not hold in positive characteristic, even for surfaces (cf. \cite{Raynaud78}, \cite{Mukai13}). So it is natural to ask to what extent Kodaira-type vanishing holds. Recently some interesting results have been established on rational surfaces. It is shown by Cascini, Tanaka and Witaszek that Kawamata-Viehweg vanishing holds for log del Pezzo surfaces if the base field is algebraically closed of sufficiently large characteristic (cf. \cite[Theorem 1.2]{CTW16}). However, Cascini and Tanaka have shown that Kawamata-Viehweg vanishing does not hold for rational surfaces defined over any algebraically closed field of positive characteristic (cf. \cite[Theorem 1.1]{CT16}).

In this short paper we prove two Kodaira-type vanishing theorems for projective toric surfaces over arbitrary fields, thus filling in another piece of the puzzle. The first result confirms that Kawamata-Viehweg vanishing holds for projective toric surfaces. 
\begin{thmx} \label{main}
Let $(X,\Delta)$ be a klt pair where $X$ is a projective toric surface. If $D$ is a $\mathbb{Z}$-divisor on $X$ such that $D-(K_X+\Delta)$ is nef and big, then $H^i(X,\mathcal{O}_X(D))=0$ for any $i>0$.
\end{thmx}

\noindent Compared to \cite[Theorem 1.2]{CTW16}, we restrict to toric surfaces, but do not have any restriction on the base field. Next, using Theorem \ref{main} we prove the following
\begin{thmx} \label{mainvar}
Let $X$ be a projective toric surface and $D$ an effective nef $\mathbb{Q}$-divisor on $X$ such that $(X,\lceil D\rceil-D)$ is klt. Then $H^i(X,\mathcal{O}_X(K_X+\lceil D\rceil))=0$ for any $i\ne 2-\kappa(X,D)$.
\end{thmx}

Note that the corresponding statements for Theorem \ref{main} and \ref{mainvar} have been proven in all dimensions by Fujino (see \cite[Corollary 1.7]{Fujino07}, and also \cite{Mustata02}), assuming that $X$ is a split toric variety and $D$ is torus invariant. Here we treat the 2-dimensional case where $X$ is not necessarily split and $D$ is not necessarily torus invariant. So far, very little is known in this direction in general. The condition that $D$ is effective in Theorem \ref{mainvar} might seem unnatural at first glance, as it is not required in \cite[Corollary 1.7]{Fujino07}. But actually Theorem \ref{mainvar} does not hold if we do not have such requirement. An easy counter-example is as follows. 
\begin{example}\label{easyctex}
Let $X=\mathbb{P}^1\times\mathbb{P}^1$ over an algebraically closed field. There are four torus-invariant divisors on $X$, namely $D_1,...,D_4$ where $D_1$ and $D_3$ are of type $(1,0)$ and $D_2$ and $D_4$ are of type $(0,1)$. By \cite[Ch. II, Example 7.6.2]{Hartshorne77} we know $D_1+mD_2$ is very ample for all $m\ge 1$. Next we take $D'\in |2(D_1+3D_2)|$ to be a smooth prime divisor. Now let $D:=3D_1+3D_2-\dfrac{1}{2}D'$, then $D\sim_{\mathbb{Q}}2D_1$ and in particular $\kappa(X,D)=1$. On the other hand we have that $0\ne 2-1$, and yet
\begin{align*}
& h^0(X,\mathcal{O}_X(K_X+\lceil D\rceil))=h^0(X,\mathcal{O}_X(-(\sum_{i=1}^4D_i)+3D_1+3D_2)) \\
=& h^0(X,\mathcal{O}_X(D_1+D_2))=4\ne 0.
\end{align*}
\end{example}
\subsection*{Acknowledgements} 
The authors would like to thank David Cox, Omprokash Das, Alexander Duncan, Honglu Fan, Christopher Hacon and Mircea Musta\cb{t}\u{a} for many useful discussions.
\section{Preliminaries} 
Throughout the paper we work over an arbitrary field $k$. We will freely use the standard notations in birational geometry which can be found in \cite[3.G]{HK10} (e.g. pairs and klt singularities). 
\begin{definition}
Let $X$ be a normal geometrically irreducible variety over $k$. The variety $X$ is called a \textit{toric variety} if there is an algebraic torus $T$ acting on $X$ and an open orbit $U$ such that $U$ is a principal homogeneous space over $T$. 
\end{definition}
\begin{definition}
We say that $X$ is \textit{split} if the torus $T\cong \mathbb{G}_{m,k}^{\times n}$ is split. 
\end{definition}
Let $k^s$ be a separable closure of $k$. By the definition of an algebraic torus, the torus $T^s:= T\otimes_k k^s\cong \mathbb{G}_{m,k^s}^{\times n}$ is split. It is well known that the torus $T$ also splits over a finite Galois field extension, and we include a proof for readers' convenience.
\begin{lemma} \label{reducetosplit}
The torus $T$ splits over a finite Galois field extension $l$ of $k$. 
\end{lemma}
\begin{proof}
Let $\phi: T^s\to \mathbb{G}_{m,k^s}^{\times n}$ be an isomorphism. Let $\Gamma=\mathrm{Gal}(k^s/k)$ be the absolute Galois group. The tori $T^s$ and $\mathbb{G}_{m,k^s}^{\times n}$ have a natural $\Gamma$-action with $\Gamma$ acting on $k^s$ (note that $\phi$ is not $\Gamma$-invariant unless $T$ is split). For each $g\in \Gamma$, the map $\phi_g:=\phi g\phi^{-1}g^{-1}$ gives a group automorphism of $\mathbb{G}_{m,k^s}^{\times n}$ over $k^s$. Note that $\mathrm{Aut}_{\text{gp}}(\mathbb{G}_{m,k^s}^{\times n})\cong\mathrm{GL}(n,\mathbb{Z})$. Thus, we get a continuous map $\chi: \Gamma\to\mathrm{GL}(n,\mathbb{Z})$ by sending $g$ to $\phi_g$ where $\Gamma$ is equipped with the profinite topology and $\mathrm{GL}(n,\mathbb{Z})$ with the discrete topology. Since $\Gamma$ is compact and $\mathrm{GL}(n,\mathbb{Z})$ is discrete, the image of $\chi$ is finite. This implies that $\ker(\chi)=\mathrm{Gal}(k^s/l)$ for a finite Galois extension $l$ of $k$. Since $\phi$ is $\mathrm{Gal}(k^s/l)$-invariant, it descends to an isomorphism $T\otimes_k l\cong\mathbb{G}_{m,l}^{\times n}$. 
\end{proof}

For split toric varieties, \cite{CLS11} is a great reference. Although \cite{CLS11} only mention toric varieties over $\mathbb{C}$, the results used in this paper from loc. cit. also apply to split toric varieties over arbitrary fields.

The main tool to prove Theorem \ref{main} and Theorem \ref{mainvar} is the minimal model program (MMP). For a split $\mathbb{Q}$-factorial projective toric variety, we can run minimal model program with respect to any divisor \cite[Procedure 15.5.5]{CLS11}. We have the following 
\begin{theorem} \label{toricMMP}
Let $X$ be a split $\mathbb{Q}$-factorial projective toric variety. Let $D$ be a divisor on $X$. Then there is a sequence of divisorial contractions and flips $X=X_0\dasharrow X_1\dasharrow ... \dasharrow X_N$, each of which is a toric map, such that if we denote the strict transform of $D$ on $X_N$ by $D_N$, then either
\begin{itemize}
\item $D_N$ is nef, or
\item there exists a fibration $g:X_N\to Z$ from $X_N$ to another toric variety $Z$ such that $\dim Z<\dim X_N$, $\rho(X_N/Z)=1$ and $-D_N$ is $g$-ample. 
\end{itemize} 
\end{theorem}
A split toric variety is $\mathbb{Q}$-factorial if and only if all cones in the fan are simplicial. This is because by \cite[Theorem 4.1.3]{CLS11}, any divisor is linearly equivalent to a $\mathbb{Z}$-linear combination of torus invariant divisors and all torus invariant divisors are $\mathbb{Q}$-Cartier if and only if all cones are simplicial. Since $1$- or $2$-dimensional cones are always simplicial, any split toric surface is $\mathbb{Q}$-factorial.

Now we point out that the cohomology of a divisor $D$ is preserved through any birational transform in a $D$-MMP of normal surfaces. This is proven in \cite[Step 3 of Proof of Theorem 1.2]{CTW16}.
\begin{lemma} \label{cohomologypreservedbirl}
Let $X$ and $Y$ be normal surfaces, $D$ a $\mathbb{Q}$-Cartier divisor on $X$, and $f: X \to Y$ a birational contraction of a curve $E$ such that $D \cdot E\le 0$. Then
$$H^k(X, \mathcal{O}_X (D)) = H^k(Y, \mathcal{O}_Y (f_*D))$$
for any $k\in \mathbb{Z}$.
\end{lemma}
We next show that bigness, nefness, klt, Kodaira dimension and vanishing of cohomology are preserved under base-change of a finite separable extension.
\begin{lemma} \label{cohpreserved}
Let $X$ be a noetherian separated scheme of finite type over $k$ and $l$ a field extension of $k$. Let $\mathcal{F}$ be a quasi-coherent sheaf on $X$ and $Y=X\times_{{\rm Spec\,}k} {\rm Spec\,}l$ with $v: Y\to X$ the induced morphism. Then we have $H^i(Y,v^*\mathcal{F})=H^i(X,\mathcal{F})\otimes_k l$ for all $i\geqslant 0$.
\end{lemma}
\begin{proof}
We have the following commutative diagram:
\begin{center}
\begin{tikzcd}
Y \arrow{r}{v} \arrow{d}{g}     & X \arrow{d}{f}\\
\mathrm{Spec}\, l \arrow{r}{u} & \mathrm{Spec}\, k
\end{tikzcd}
\end{center}
By \cite[Ch. III, Proposition 8.5, 9.3]{Hartshorne77}, we have 
$$H^i(Y,v^*\mathcal{F})=R^ig_*(v^*\mathcal{F})=u^*R^if_*(\mathcal{F})=H^i(X,\mathcal{F})\otimes_k l.$$
\end{proof}
Lemma \ref{cohpreserved} shows that the sheaf cohomology is preserved under base-change of fields through the pull-back of sheaves. In particular, we have that bigness, Kodaira dimension and vanishing of cohomology are preserved under such change.

Nefness is also preserved by the following lemma as the projection $Y=X\times_{{\rm Spec\,}k} {\rm Spec\,}l\to X$ for $l$ a finite field extension of $k$ is a finite morphism.
\begin{lemma} \label{nefpreserved}
Let $f: Y\to X$ be a finite morphism. If $D$ is a nef divisor on $X$, then $f^*D$ is a nef divisor on $Y$.
\end{lemma}
\begin{proof}
Let $C$ be any curve on $Y$. Then by projection formula (cf. \cite[Proposition 2.3 (c)]{Fulton98}, note that a finite morphism is proper), $f_*(f^*D.C)=D.f_*C\geqslant 0$. Since $f_*(f^*D.C) = [k(Y):k(X)] f^*D.C$, we have $f^*D.C\geqslant 0$.
\end{proof}
Now klt is preserved by the following
\begin{lemma} \label{kltpreserved}
Let $(X,\Delta)$ be a 2-dimensional projective klt pair over a field $k$. Let $l$ be a finite separable field extension of $k$ and $Y=X\times_{{\rm Spec\,}k} {\rm Spec\,}l$ with $v: Y\to X$ the induced morphism. Then $(Y, v^*\Delta)$ is klt.
\end{lemma}
\begin{proof}
Applying \cite[Ch. II, Proposition 8.11 A]{Hartshorne77} to $Y\overset{g}{\to} {\rm Spec\,}l \to {\rm Spec\,}k$, we have an exact sequence $g^*\Omega_{l/k}^1\to\Omega_{Y/k}^1\to\Omega_{Y/l}^1\to 0$. By \cite[Ch. II, Theorem 8.6 A and Proposition 8.10]{Hartshorne77}, we have $\Omega_{l/k}^1=0$ and thus $\Omega_{Y/k}^1\cong\Omega_{Y/l}^1\cong v^*\Omega_{X/k}^1$. Hence, $v^*K_{X/k}=K_{Y/k}$ and the ramification divisor of $v$ is $0$. So by \cite[Corollary 2.43 (1)]{Kollar13} we are done ($v^*\Delta$ here corresponds to $\Delta'$ in \cite[Corollary 2.43 (1)]{Kollar13}).
\end{proof}
Finally in this section we show that round-up and pull-back of $\mathbb{Q}$-divisors by a finite Galois field extension commutes with each other.
\begin{lemma} \label{rounddownpreserved}
Let $l$ be a finite Galois field extension of $k$. Let $X$ be a noetherian normal integral separated scheme over $k$, $Y=X\times_{{\rm Spec\,}k} {\rm Spec\,}l$ and $v: Y\to X$ be the natural projection. Let $D$ be a $\mathbb{Q}$-divisor on $X$, then $\lceil v^*D\rceil=v^*\lceil D\rceil$.
\end{lemma}
\begin{proof}
It suffices to show that (i) if $D$ is a prime ($\mathbb{Z}$-)divisor on $X$, then all irreducible components of $v^{-1}D$ have geometric multiplicities one; (ii) if $D'$ is another prime divisor on $X$ that is different from $D$, then components of $v^{-1}D'$ are distinct from components of $v^{-1}D$.

Assume $v^{-1}D$ and $v^{-1}D'$ have a common component $E$. Since $D$ and $D'$ are prime divisors, we have $D=D'=v(E)$, which is a contradiction. This proves (ii). To prove (i), we first note that by Galois descent, a closed subscheme $S$ on $Y$ is of the form $v^*L$ for some closed subscheme $L$ on $X$ (we say $S$ descends to $L$ in this case) if and only if $S$ is $G={\rm Gal}(l/k)$-invariant. Now let $v^*(D)=\sum_{i=1}^n a_i D_i$ where $D_i$ are irreducible components of $v^{-1}D$ and $a_i$ are geometric multiplicities of $D_i$, we have $a_i\in\mathbb{Z}$ and $a_i\geqslant 1$. Since $\sum a_i D_i$ is $G$-invariant, $G$ permutes $D_i$. Assume $\{D_1,...,D_m\}$ is the $G$-orbit for $D_1$, then $\sum_{i=1}^m D_i$ descends to a divisor $D'$ on $X$. Since $D$ is irreducible and reduced, we must have $D'=D$. Thus, $m=n$ and $a_i=1$.
\end{proof}
\section{Proof of the theorems} 
We first reduce the problem to the case of split toric varieties.
By Lemma \ref{cohpreserved}, \ref{nefpreserved}, \ref{kltpreserved} and \ref{rounddownpreserved} from the previous section, all conditions involved in Theorem \ref{main} and \ref{mainvar} are compatible with the base-change by a finite Galois field extension. Therefore, by Lemma \ref{reducetosplit}, in order to prove Theorem \ref{main} and \ref{mainvar}, we can assume that $X$ is split. 

For the rest of the proof we assume that $X$ is a split toric variety. We first present the following lemma, which is just a rephrasing of a special case of \cite[Corollary 1.7]{Fujino07}.
\begin{lemma}\label{Fujivar}
Let $(X,\Delta)$ be a klt pair where $X$ is a toric variety and $\Delta$ is a torus-invariant divisor. Let $D$ be a $\mathbb{Z}$-divisor on $X$ such that $D-(K_X+\Delta)$ is nef and big. Then $H^i(X,\mathcal{O}_X(D))=0$ for $i>0$. 
\end{lemma}
\begin{proof}
By \cite[Theorem 4.1.3]{CLS11}, $D$ is linearly equivalent to a linear combination of torus-invariant divisors with integral coefficients: $D\sim_{\rm lin}\sum_ia_iD_i$. Then $D\sim_{\rm lin}K_X+(\sum_ia_iD_i-K_X)=K_X+\lceil \sum_ia_iD_i-(K_X+\Delta) \rceil$, where the second equality is because the coefficients in $\Delta$ are in $[0,1)$. Therefore the vanishing of $H^i(X,\mathcal{O}_X(D))=0$ is just a consequence of \cite[Corollary 1.7]{Fujino07}.
\end{proof}
\begin{proof}[Proof of Theorem \ref{main}]
The idea is similar to that of \cite[Theorem 1.2]{CTW16}. If $\Delta=0$ this is just Lemma \ref{Fujivar}.

Next we prove the theorem under the assumption that $D$ is nef. Since $X$ is a toric variety, there is a torus-invariant $\mathbb{Q}$-divisor $B$ such that $(X,B)$ is klt and $-(K_X+B)$ is ample (see. \cite[Example 11.4.26]{CLS11}). Now 
$$D=K_X+B+(-(K_X+B)+D)$$
where $-(K_X+B)+D$ is ample. Therefore the vanishing of $H^i(X,\mathcal{O}_X(D))$ is still Lemma \ref{Fujivar}.

We now show that we can assume that there exists a $D$-negative Mori fiber space $g:X\to Z$ onto a smooth projective curve $Z$. By Theorem \ref{toricMMP} we can run a $D$-MMP for $X$ as 
$$X=X_0 \overset{\text{$f_0$}}\to X_1\overset{\text{$f_1$}}\to X_2\overset{\text{$f_2$}}\to ...$$ 
In particular, every birational map $f_i:X_i\to X_{i+1}$ during the minimal model program is a toric map. We denote the strict transform of $D$ on $X_i$ by $D_i$. Then by Lemma \ref{cohomologypreservedbirl} we have $H^k(X, \mathcal{O}_X (D)) = H^k(X_i, \mathcal{O}_{X_i} (D_i))$. Therefore we can directly assume that one of the following is true: 
\begin{enumerate}
\item $D$ is nef.
\item $\rho(X)=2$ and there exists a $D$-negative Mori fiber space $g:X\to Z$ onto a smooth projective curve $Z$. 
\item $\rho(X)=1$ and $-D$ is ample.
\end{enumerate}
(1) is done in the second paragraph of the proof. In (3) since $\rho(X)=1$, there is only one numerical divisor class on $X$, and in particular $\Delta$ is ample. Then by the assumption that $D-(K_X+\Delta)$ is nef and big we actually have that $D-K_X$ is ample. This means that we can directly assume that $\Delta=0$, and this is done by Lemma \ref{Fujivar} as in the first paragraph. So for the rest of the proof we assume that we are in Case (2).

Finally, following exactly \cite[Proof of Theorem 1.2, Step 4]{CTW16} we may assume that $\Delta=0$, and this is again Lemma \ref{Fujivar}.
\end{proof}
\begin{lemma}\label{Dbig}
Let $X$ be a projective toric variety. Let $D$ be a nef and effective $\mathbb{Q}$-divisor on $X$. If $K_X+\lceil D\rceil$ is nef then $D$ is also big.
\end{lemma}
\begin{proof}
We know that $-K_X$ is big as $X$ is a projective toric variety (see \cite[Example 11.4.26]{CLS11}). So $K_X+\lceil D\rceil$ being nef implies that $\lceil D\rceil$ is big. Since we can choose an $\epsilon>0$ such that $\epsilon\lceil D\rceil\le D$ we see that $D$ is also big.
\end{proof}
\begin{proof}[Proof of Theorem \ref{mainvar}]
If $D=\lceil D\rceil$ this is just \cite[Corollary 1.7]{Fujino07}.

Next we prove the theorem under the assumption that  $K_X+\lceil D\rceil$ is nef. By Lemma \ref{Dbig} we know that $D$ is big, in particular $\kappa(X,D)=2$. Then the theorem follows from Theorem \ref{main}. 

Now we run a $(K_X+\lceil D\rceil)$-MMP for $X$. By Theorem \ref{toricMMP} and Lemma \ref{cohomologypreservedbirl} as in the proof of Theorem \ref{main}, we can assume that one of the following holds.
\begin{enumerate}[label=(\alph*),leftmargin=3\parindent]
\item $K_X+\lceil D\rceil$ is nef.
\item $\rho(X)=2$ and there exists a $(K_X+\lceil D\rceil)$-negative Mori fiber space $g:X\to Z$ onto a smooth projective curve $Z$. 
\item $\rho(X)=1$ and $-(K_X+\lceil D\rceil)$ is ample.
\end{enumerate}
(a) is done in the above paragraph. In (c) since $\rho(X)=1$, $D$ is either $0$ or ample, so this is just Theorem \ref{main}. Therefore we can assume that we are in the second case. 

If every curve on $X$ is nef, then $\lceil D\rceil$ is also nef and by \cite[Corollary 1.7]{Fujino07} we are done. Thus we may assume that there exists a curve $E$ on $X$ such that $E^2<0$, and every curve different from $E$ is nef. Therefore, after possibly increasing the coefficients of nef components of $D$ we may assume that $\lceil D\rceil=D+\delta E$ for some $\delta\in [0,1)$. 
Now if $\delta=0$ then $\lceil D\rceil=D$ and we are done. If $\delta\ne 0$ then we can write $D$ as $D=D'+\delta'E$ where $D'$ and $E$ have no component in common, $\delta'\ne 0$ and $D'\ne 0$ (if $D'=0$ then $D$ is not nef). But then $D$ must be big as well and by Theorem \ref{main} we are done.
\end{proof}
\bibliographystyle{alpha}
\bibliography{P}  
\end{document}